\documentclass[12pt]{amsart}
\usepackage{amsfonts}
\usepackage{epsfig}
\usepackage{graphicx}
\usepackage{amsmath}
\usepackage{amssymb}
\usepackage{latexsym}
\usepackage{rotating}
\usepackage{pstricks, pst-node, pst-text, pst-3d}
\usepackage{amsbsy}
\usepackage{subfig}
\usepackage{amsthm}
\usepackage{mathrsfs}
\usepackage[all]{xy}
\usepackage{a4wide}
\usepackage{hyperref}
\usepackage{color}
\usepackage{verbatim}
\input{amssym.def}
\newtheorem{theorem}{Theorem}
\newtheorem{proposition}{Proposition}

\newtheorem{definition}{Definition}
\theoremstyle{remark}

\DeclareMathOperator{\re}{Re}
\DeclareMathOperator{\im}{Im}

\DeclareMathOperator{\arctanh}{arctanh}

\DeclareMathOperator{\id}{id}
\DeclareMathOperator{\koebe}{K}

\DeclareMathOperator{\dist}{dist}
\begin{document}
\title[Boundary L\"owner evolution]{L\"owner evolution driven by a stochastic boundary point}
\author{Georgy Ivanov and Alexander Vasil'ev}

\address{ 
\noindent \newline Department of Mathematics\newline
University of Bergen \newline P.O.~Box~7803 \newline Bergen N-5020, Norway \bigskip}

\email{\newline georgy.ivanov@math.uib.no \newline alexander.vasiliev@math.uib.no}

\thanks{The authors have been  supported by the grant of the Norwegian Research Council \#204726/V30, by the NordForsk network `Analysis and Applications', grant \#080151,  by the European Science Foundation Research Networking Programme HCAA, and by Meltzerfondet (University of Bergen). This work was completed while the authors were visiting Mittag-Leffler institute, Sweden in the Fall 2011.}



\begin{abstract}
We consider evolution in the unit disk in which the sample paths
are represented by the trajectories of points evolving randomly under
the generalized Loewner equation. The driving mechanism differs from
the SLE evolution, but nevertheless solutions possess similar invariance properties.
\end{abstract}

\maketitle

\section{Introduction}\label{introd}

The last decade has been marked by a burst of interest in Schramm-L\"owner Evolution (SLE). SLE has led to an elegant description of several 2D conformally invariant statistical physical systems at criticality by means of a solution to the initial value problem for a special differential equation with
a random driving term given by 1D Brownian motion. The origin of SLE can be traced to the seminal Schramm's paper \cite{Schramm}, where he revisited the notion of scaling limit and conformal invariance for the loop erased random walk and the uniform spanning tree. This approach was developed in many
further works and many lattice physical processes were proved to converge to SLE with some specific diffusion factor, e.g., percolation, Ising model \cite{SLE1, SLE2, Smirnov}.  On the other hand, SLE became an approach to conformal field theory which emphasizes CFT's roots in statistical physics \cite{BB, FriedrichWerner}.

We consider another setup in which  the sample paths are represented by the trajectories of  points (e.g., the origin) in the unit disk $\mathbb D$ evolving randomly under the generalized L\"owner equation. The driving mechanism differs from SLE. In the SLE case the Denjoy-Wolff attracting point ($\infty$ in the chordal case or a boundary point of the unit disk in the radial case) is fixed. In our case, the attracting point is the driving mechanism and the Denjoy-Wolff point is different from it. In analytic terms, the generalized L\"owner  evolution $\{\phi_t\}_{t\geq 0}$ solves the initial value problem
\begin{equation}\label{LEV}
 \begin{cases}
  \frac{d}{dt}\phi_t(z) = (\phi_t(z)-\tau(t))(\overline{\tau(t)} \phi_t(z) - 1)\, p(\phi_t(z),t),\\
\phi_0(z) = z,
 \end{cases} \quad z\in \mathbb{D},\quad t\geq 0,
\end{equation}
where $\tau: [0,+\infty) \to \overline{\mathbb{D}}$ is measurable and the driving function $p(z,t)$ is a so-called Herglotz function (see Section \ref{sec:genloew} for the definition). 

In the radial SLE case $\tau\equiv 0$ and 
\[
p(z,t)=\frac{e^{ikB_t}+z}{e^{ikB_t}-z}, \quad\text{where $B_t$ is Brownian motion.}
\]
In our case $\tau(t)=e^{ikB_t}$, and the Denjoy-Wolff point is different from $\tau$, and there is no common fixed point of the evolution.

Even though the driving mechanism in our case differs from that of SLE, the generated families of conformal maps still possess the important  time-homogeneous Markov property.

In the deterministic case, Section 3, we thoroughly study a particular case when $\tau(t)=e^{ikt}$, $k\in\mathbb R$. We give a complete description of the case of the evolution of M\"obius automorphisms of $\mathbb D$. In particular, we obtain the intervals of the constant $k$ in which the evolution is elliptic, hyperbolic and parabolic. In the general case of endomorphisms of $\mathbb D$ we prove a sufficient condition of ellipticity of the dynamics for sufficient large $k$. We also find values of the parameter $k$ for which the trajectories of a point are closed.

The proposed model also describes deterministic and stochastic evolution of shapes in the complex plane.
Following \cite{SharonMumford}, we understand  by a {\it shape}  a  simple closed smooth curve. The study of 2D shapes is one of the central problems in the field of computer vision. A program of such study and its importance was summarized by Mumford at ICM 2002 in Beijing \cite{Mumford}. However, we plan to focus on shape evolution in forthcoming papers.

The central part of the paper, Section 4, is dedicated to a stochastic version of \eqref{LEV} in which the driving term is $\tau(t)=e^{ikB_t}$, where $B_{t}$ is the standard 1-dimensional Brownian motion. The equation \eqref{LEV} is reduced to an It{\^o} 2D diffusion SDE,  which we solve explicitly
in some cases. We calculate the infinitesimal generator which is given by the  Virasoro generators in a certain representation.  We give estimates of growth of $|\phi_t(z)|$, calculate momenta, and describe the boundary diffusion on $\mathbb T$.

So far, the only well-studied examples of the generalized L\"owner equation were the classical cases corresponding to $\tau(t) \equiv 0$ (radial L\"owner equation) and $\tau(t) \equiv 1$ (chordal L\"owner equation). This paper is one of the first studies where examples of generalized L\"owner equation with non-constant attracting point $\tau(t)$ are considered.

\medskip
\noindent
{\bf Acknowledgments.} The authors would like to thank Filippo Bracci and  David Shoikhet for inspiring discussions. The authors are especially grateful to Alan Sola for reading the original version of the paper and for his important critical remarks.

\section{Generalized L\"owner evolution}
\label{sec:genloew}

The pioneering idea of L\"owner \cite{Loewner} in 1923 contained two main ingredients: subordination chains and semigroups of conformal maps. This far-reaching program was created in the hopes to solve
the Bieberbach conjecture \cite{Bieb} and the final proof of this conjecture by de~Branges \cite{Branges} in 1984 was based on L\"owner's  parametric method.
The modern form of this method is due to Kufarev \cite{Kufarev} and Pommerenke \cite{Pommerenke1, Pommerenke2}.  We concentrate our attention on 
a generalization of the L\"owner-Kufarev approach by  F.~Bracci, M.~D. Contreras, S.~D{\'i}az-Madrigal, and P.~Gumenyuk  \cite{Bracci, gumenyuk10}. One of the main results of the generalized L\"owner theory is an essentially one-to-one correspondence between  evolution families, Herglotz vector fields and (generalized) L\"owner chains. Below we briefly describe the basic notions of the theory.

\subsection{Preliminaries: Semigroups and fixed points of holomorphic functions}

By the Schwarz-Pick lemma, every holomorphic self-map $\varphi$ of the unit disk $\mathbb{D}$ may have at most one fixed point $\tau$ in $\mathbb{D}.$ If such a point $\tau$ exists, then the point is called \emph{the Denjoy-Wolff point} of $\varphi,$ and in this case the function $\varphi$ is said to be an elliptic self-map of $\mathbb{D}.$ 

Otherwise, there exists a point $\tau$ on the unit circle $ \mathbb{T},$ such that the angular limit $\angle \lim_{z\to \tau} \varphi(z) = \tau$ (this statement is known as the Denjoy-Wolff theorem).  The point $\tau$ is again called \emph{the Denjoy-Wolff point} of $\varphi$. The limit $\angle \lim_{z\to\tau}\varphi'(z) = \alpha_f$ always exists, its value $\alpha_\varphi$ belongs to $(0,1],$ and the map $\varphi$ in this case is said to be either \emph{hyperbolic} (if $\alpha_\varphi < 1$) or \emph{parabolic} (if $\alpha_\varphi=1$) (for details and proofs see, e.~g., \cite{Abate, ShoikhetSeimgroupsGFT}).

A family $\{\psi_t\}_{t\geq 0}$ of holomorphic self-maps of the the unit disk $\mathbb{D}$ is called a one-parameter continuous semigroup if 
\begin{enumerate}
 \item $\psi_0 = \id_{\mathbb{D}},$
\item $\psi_{t+s} = \psi_t \circ \psi_s,$ for $s,t \geq 0,$
\item $\lim_{t\to s} \psi_t(z) = \psi_s(z),$ for all $s\geq 0$ and $z\in \mathbb{D},$
\item $\lim_{t\to 0^{+}} \psi_t(z) = z$ locally uniformly in $\mathbb{D}.$
\end{enumerate}

For every semigroup $\{\psi_t(z)\}_{t\geq 0}$ there exists a unique holomorphic function $f:\mathbb{D} \to\mathbb{C},$ such that the semigroup $\{\psi_t(z)\}_{t\geq 0}$ is the unique solution to the initial-value problem
\begin{equation}
 \begin{cases}
  \frac{d}{dt}\psi_t(z) = f(\psi_t(z)),\\
\psi_0(z) = z.
 \end{cases}
\end{equation}
The function $f$ is called the infinitesimal generator of the semigroup $\{\psi_t \}_{t\geq 0}$.

A holomorphic function $f$ is an infinitesimal generator of a semigroup if and only if there exists $\tau\in \overline{\mathbb{D}}$  and a holomorphic function $p:\mathbb{D} \to \mathbb{C}$ with $\re p \geq 0,$ such that
\[
 f(z) = (z- \tau)(\bar{\tau} z - 1) p(z),\quad z\in\mathbb{D}.
\]
This representation is unique (if $f \not\equiv 0$), and is known as the Berkson-Porta representation of $f$.  The point $\tau$ turns out to be the Denjoy-Wolff point of all functions in $\{\psi_t(z)\}_{t\geq 0}.$ We will use the term \emph{Herglotz function} for the function $p(z)$.

We say that a semigroup $\psi_t(z)$ is elliptic, hyperbolic or parabolic if it consists of elliptic, hyperbolic or parabolic self-maps of $\mathbb{D},$ respectively.

\subsection{Evolution families and Herglotz vector fields}
We start with the notion of \emph{evolution family}.
\begin{definition}\label{evolution}
An evolution family of order $d\in[1,+\infty]$ is a two-parameter family \linebreak $\{\phi_{s,t}\}_{0\leq s\leq t<+\infty}$ of holomorphic self-maps of the unit disk, such that the following three conditions are satisfied.
\begin{itemize}
 \item $\phi_{s,s} = id_{\mathbb{D}}$;
 \item $\phi_{s,t} = \phi_{u,t} \circ \phi_{s,u}$ for all $0 \leq s \leq u \leq t <+\infty$;
 \item for any $z \in \mathbb{D}$ and $T > 0$ there is a non-negative function $k_{z,T} \in L^d([0,T],\mathbb{R}),$ such that
\[
 |\phi_{s,u}(z) - \phi_{s,t}(z)| \leq \int_u^t k_{z,T}(\xi) d\xi, \quad z\in \mathbb{D}
\]
for all $0 \leq s \leq u \leq t \leq T.$
\end{itemize}
\end{definition}

An infinitesimal description of an evolution family is given in terms of a {\it Herglotz vector field}.

\begin{definition}
 A (generalized) Herglotz vector field of order $d$ is a function $G:\mathbb{D} \times [0,+\infty) \to \mathbb{C}$  satisfying the following conditions:
\begin{itemize}
 \item the  function $[0, +\infty) \ni t \mapsto G(z,t)$ is measurable for every $z\in \mathbb{D}$;
\item the function $z \mapsto G(z,t)$ is holomorphic in the unit disk for  $t \in [0,+\infty)$;
\item for any compact set $K \subset \mathbb{D}$ and for every  $T>0$ there exists a non-negative function $k_{K,T} \in L^d([0,T],\mathbb{R}),$ such that
\[
 |G(z,t)|\leq k_{K,T}(t)
\]
for all $z\in K,$ and for  almost every $t \in [0,T]$;
\item for almost every $t\in[0,+\infty)$ the vector field $G(\cdot, t)$ is semicomplete.
\end{itemize}
\end{definition}
By semicompleteness we mean that the solution to the problem
\[
 \begin{cases}
 \displaystyle \frac{dx(l)}{d l} = G(x(l),t),\\
x(s)=z
 \end{cases}
\]
is defined for all times $l\in[s,+\infty),$ for any fixed $s\geq 0,$ fixed $t \geq 0$ and fixed $z\in \mathbb{D}.$

An important result of general L\"owner theory is the fact that the evolution families can be put into a one-to-one correspondence with the Herglotz vector fields by means of the so-called generalized L\"owner ODE. This can be formulated as the following theorem.
\begin{theorem}[{\cite[Theorem 1.1]{Bracci}}]
\label{thm:LoewnerEq}
For any evolution family $\{\phi_{s,t}\}$ of order $d \geq 1$ in the unit disk there exists an essentially unique Herglotz vector field $G(z,t)$ of order $d,$ such that for all $z\in \mathbb{D}$ and for almost all $t\in[0,+\infty)$
\begin{equation}
\label{eq:evolution}
 \frac{\partial \phi_{s,t}(z)}{\partial t} = G(\phi_{s,t}(z),t).
\end{equation}
Conversely, for any Herglotz vector field $G(z,t)$ of order $d\geq 1$ in the unit disk there exists a unique evolution family of order $d$, such that the equation above is satisfied.
\end{theorem}
Essential uniqueness in the theorem above means that for any other Herglotz vector field $H(z,t)$ satisfying \eqref{eq:evolution}, the equality $H(z,t) = G(z,t)$ holds for all $z\in\mathbb{D}$ and almost all $t\in[0,+\infty)$.

Herglotz vector fields admit a convenient Berkson-Porta-like representation using so-called Herglotz functions.

\begin{definition}
 A Herglotz function of order $d\in [1,+\infty)$ is a function ${p:\mathbb{D} \times [0,+\infty) \to \mathbb{C}},$ such that
\begin{itemize}
 \item the function $t\mapsto p(z,t)$  belongs to $L^d_{\mathrm{loc}}([0,+\infty),\mathbb{C})$ for all $z\in\mathbb{D}$;
\item the function $z \mapsto p(z,t)$ is holomorphic in $\mathbb{D}$ for every fixed $t\in[0,+\infty)$;
\item $\re p(z,t) \geq 0$ for all $z\in \mathbb{D}$ and for all $t \in[0,+\infty)$. 
\end{itemize}
\end{definition}

Now, the representation of Herglotz vector fields is given in
 the following theorem.
\begin{theorem}[{\cite[Theorem 1.2]{Bracci}}]
Given a Herglotz vector field of order $d \geq 1$ in the unit disk, there exists an essentially unique (i.~e., defined uniquely for almost all $t$ for which $G(\cdot, t)\neq 0$) measurable function $\tau: [0,+\infty) \to \overline{\mathbb{D}}$ and a Herglotz function $p(z,t)$ of order $d,$ such that for all $z\in \mathbb{D}$ and almost all $t\in[0,+\infty)$
\begin{equation}
\label{eq:genLoewner}
 G(z,t) = (z-\tau(t))(\overline{\tau(t)}z-1)p(z,t).
\end{equation}
Conversely, given a measurable function $\tau:[0,+\infty) \to \overline{\mathbb{D}}$ and a Herglotz function $p(z,t)$ of order $d\geq 1,$ the vector field defined by the formula above is a Herglotz vector field of order $d$.
\end{theorem}
According to Theorem~\ref{thm:LoewnerEq}, to every evolution family $\{\phi_{s,t}\}$ one can associate an essentially unique Herglotz vector field $G(z,t)$. The pair of functions $(p,\tau)$ representing the vector field $G(z,t)$ is called the {\it Berkson-Porta data} of the evolution family $\{\phi_{s,t}\}.$

\begin{comment}

To explain the geometrical meaning of the function $\tau(t)$ we need first to remind the notion of the Denjoy-Wolff point of a unit disk .

A classical result by Denjoy and Wolff states that for a holomorphic self-map
$f$ of the unit disk $\mathbb{D}$ other than a (hyperbolic) rotation,
there exists a unique fixed point $\tau$ in the closure $\hat{\mathbb{D}}$ of $\mathbb D$, such that the
sequence of iterates $(f_{n}(z))$ converges locally uniformly on
$\mathbb{D}$ to $\tau$ as $n\rightarrow\infty$. This point $\tau$ is called the
Denjoy-Wolff point of $f$ and it is also characterized as the only fixed
point of $f$ satisfying $f^{\prime}(\tau)\in\mathbb{D}$. In other
words, $\tau$ is the only attracting fixed point of $f$ in the above
multiplier sense. It follows from the hyperbolic metric principle that, if
$f$ is not the identity, there can be no other fixed points in
$\mathbb{D}$ except the Denjoy-Wolff point but, nevertheless, $f$ can
have many other repulsive or non-regular boundary fixed points.

If $\tau\in\mathbb D$, then the endomorphism $f$  is called {\it elliptic}. Otherwise, the angular limit $\angle \lim_{z\to \tau} f(z) = \tau$ exists as well as
the andular derivative $\angle \lim_{z\to\tau}f'(z) = \alpha_f$. If the value $\alpha_f\in (0,1],$ then the map $f$ in this case is said to be either {\it hyperbolic} (if $\alpha_f < 1$) or {\it parabolic} (if $\alpha_f=1$) (for details and proofs see, e.~g., \cite{Abate89}).
\end{comment}

Let $\{\phi_{s,t}\}$ be an evolution family with Berkson-Porta data $(p,\tau)$.
In the simplest case when neither $p$, nor $\tau$ change in time (i.~e., the corresponding Herglotz vector field $G(z,t)$ is time-independent), $\tau$ turns out to be precisely the Denjoy-Wolff point of every self-map in the family $\{\phi_{s,t}\}.$ Moreover, for any ${0 \leq s < +\infty},$ we have that $\phi_{s,t}(z) \to \tau $ uniformly on compacts subsets of $\mathbb{D}$, as $t\to +\infty$. For this reason, we call $\tau$ the {\it attracting point} of the evolution family $\{\phi_{s,t}\}.$

In the case when the Herglotz field $G(z,t)$ is time-dependent, the meaning of $\tau$ is explained in the following theorem.

\begin{theorem}[{\cite[Theorem 6.7]{Bracci}}]
 Let  $\{\phi_{s,t}\}$ be an evolution family of order $d\geq 1$ in the unit disk, and let $G(z,t) = (z- \tau(t))(\overline{\tau(t)}z-1)p(z,t)$ be the corresponding Herglotz vector field. Then for almost every $s \in[0,+\infty)$, such that  $G(z,s)\neq 0,$ there exists a decreasing sequence $\{t_n(s)\}$ converging to $s,$ such that $\phi_{s,t_n(s)} \neq id_{\mathbb{D}}$ and
\[
 \tau(s) = \lim_{n \to \infty}\tau(s,n),
\]
where $\tau(s,n)$ denotes the Denjoy-Wolff point of $\phi_{s,t_n(s)}.$
\end{theorem}

Note, however, that it is \textbf{not} true in general that $\tau(t)$ is the Denjoy-Wolff point of the map $\phi_{s,t}(z)$ (see, for instance, Example 1 in Section \ref{sec:deterministic}).

\section{Deterministic case}
\label{sec:deterministic}
%First we will study deterministic evolution families and L\"owner chains with the attracting point $\tau(t)$ moving along the boundary $\mathbb T$ of $\mathbb D$ preparing some relevant information to the stochastic case. 
We start by considering a simpler deterministic setting  $\tau(t) = e^{ikt}$. In the next Section~\ref{sec:stochastic} we switch to the stochastic setting and let $\tau(t)$ move as a Brownian particle on the unit circle $\mathbb{T}$, i.e.,  $\tau(t)= e^{ikB_t},$ where $k\geq 0,$ and $B_t$ is a 1-dimensional standard Brownian motion.

Due to the definition of an evolution family (Definition~\ref{evolution}), $\phi_{s,t} = \phi_t \circ \phi^{-1}_s,$ and one may restrict attention to the one-parameter family of maps $\{\phi_t\}_{t\geq 0} := \{\phi_{0,t}\}_{t \geq 0},$ rather than consider the two-parameter evolution family.  So we are going to analyze the following initial value problem
\begin{equation}
\label{eq:deterministic}
 \begin{cases}
  \frac{d}{dt}\phi_t(z) = \frac{(\tau(t) - \phi_t(z))^2}{\tau(t)}\, p(\phi_t(z),t),\\
\phi_0(z) = z,
 \end{cases} \quad \textrm{where } \tau(t) = e^{ikt},\quad z\in \mathbb{D},\quad t\geq 0.
\end{equation}

If the Herglotz function $p(z,t)$ can be written in the form $p(z,t) = \tilde{p}(z/\tau(t)),$ where  $\tilde{p}:\mathbb{D} \to \mathbb{C}$ is a holomorphic function with non-negative real part, then the problem (\ref{eq:deterministic}) is simplified significantly. Indeed, the change of variables
\[
 \psi_t(z)=\frac{\phi_t(z)}{\tau(t)}
\]
leads to the following initial value problem for a separable differential equation
\begin{equation}
\label{eq:autonomLoewner}
\begin{cases}
 \frac{d}{dt}\psi_t(z) = (\psi_t(z)-1)^2 \tilde{p}(\psi_t(z)) - ik\psi_t(z),\\
\psi_0(z)=z,
\end{cases}
\end{equation}
so that  the solution to (\ref{eq:deterministic}) can be written in quadratures.

Observe that the vector field $f(z) = (z-1)^2 \tilde{p}(z) - ikz$ in the right-hand side of (\ref{eq:autonomLoewner}) is an infinitesimal generator, and (\ref{eq:autonomLoewner})  is, in fact, an autonomous L\"owner equation. It generates a semigroup $\{\psi_t\}_{t\geq 0}$ of self-maps of the unit disk. We denote the Berkson-Porta data associated with this semigroup by $(p_0(z),\tau_0)$.

If  $\tau_0$ lies inside $\mathbb{D},$  then $f(\tau_0) = 0.$ Otherwise, if $\tau_0$ is on the circle $\mathbb{T},$ then  we have $\angle \lim_{z \to \tau_0} f(z) = 0.$  Hence, the corresponding Herglotz function in the Berkson-Porta decomposition of $f$ is  given by
\[
 p_0(z) = \frac{(z-1)^2 \tilde{p}(z) -ikz}{(z-\tau_0)(\overline{\tau_0} z -1)}.
\]

\subsubsection*{Example 1} The simplest  example is given by the Herglotz function $\tilde{p}(w)=\frac{1}{1-w}$, which corresponds to $p(z,t) = \frac{\tau(t)}{\tau(t)-z}$. The L\"owner ODE becomes linear
\[
 \frac{d}{dt} \phi_t (z) = \tau - \phi_t(z),
\]
and the solution to the problem (\ref{eq:deterministic}) is given by
\[
 \phi_t(z) = e^{-t} z + \frac{e^{ikt}-e^{-t}}{1+ik}.
\]

Let $\mathrm{DW}_{\phi}(t)$ denote the Denjoy-Wolff point of the map $\phi_t.$ Solving the equation $\phi_t(z) = z,$ we easily get that in this example
\[
 \mathrm{DW}_{\phi}(t) = \frac{-1+e^{t+i k t}}{(1+i k) \left(e^t-1\right)}.
\]
This may be regarded as a simple example showing that in general $\tau(t) \neq \mathrm{DW}(t).$

The corresponding autonomous L\"owner equation is
\[
 \frac{d}{dt}\psi_t(z) = 1 - (1+ik)\psi_t(z),
\]
and its solution is
\[
 \psi_t(z) = e^{-t(1+ik)}z  + \frac{1 - e^{-t(1+ik)}}{1+ik}.
\]

Let us analyze the dynamical properties of this evolution.  The Denjoy-Wolff point of the evolution $\psi_t(z)$ is $\tau_0 = \frac{1}{1+ik}$. The point $\tau_0$ is hyperbolic if $k=0$, and elliptic otherwise. The Herglotz function in the Berkson-Porta decomposition of $f$ is
\[
 p_0 (z) = -\frac{1+k^2}{z-(1-ik)}.
\]
\subsubsection*{Example 2}

In the previous example, the solution is given by a family of M\"obius transformations (in particular, circles are always transformed into circles). If one is interested in the evolution of shapes, it could be more relevant to look at examples where $p(z,t)$ is a transcendental analytic function with non-negative real part.  Figure \ref{fig:expevolution} shows the evolution of the domains $\phi_t(\mathbb{D})$ corresponding to the function $p(z,t) = \exp\frac{\pi z}{2 \tau}$.
\begin{figure}[h]
  \centering
  \subfloat[$t = \frac{\pi}{4}$]{\includegraphics[width=0.3\textwidth]{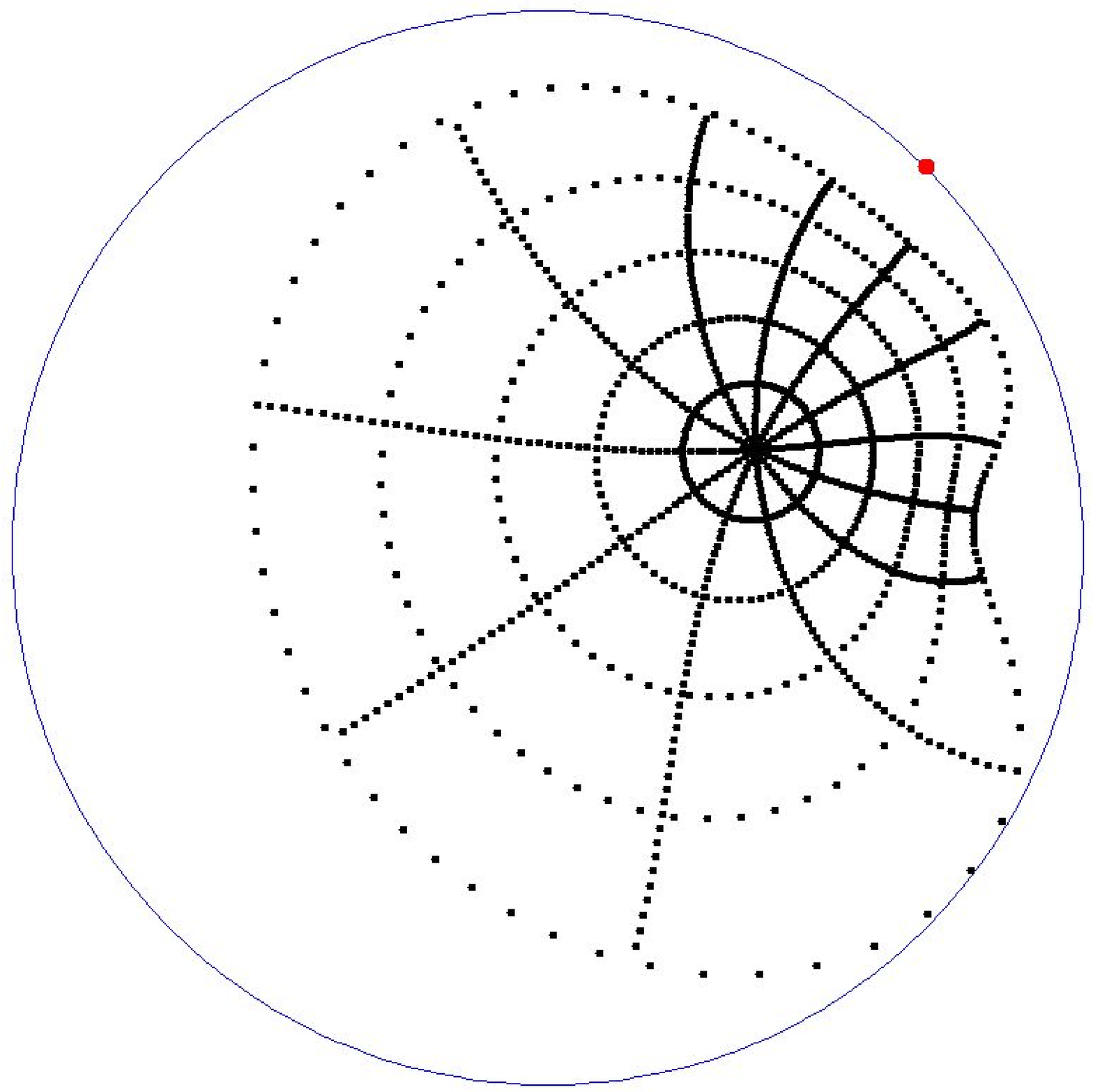}}  \hspace{5pt} \subfloat[$t = \frac{\pi}{2}$]{\includegraphics[width=0.3\textwidth]{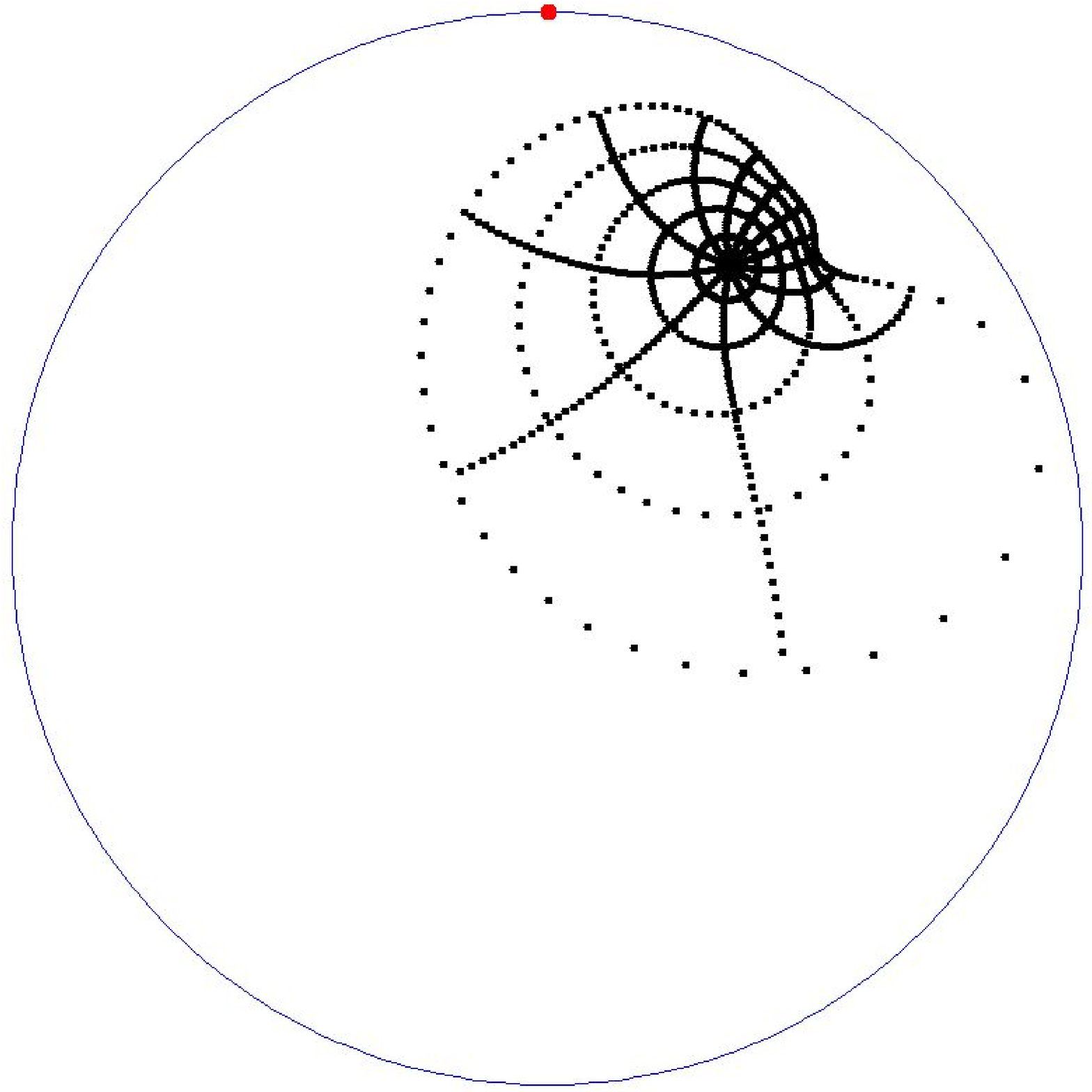}}  \hspace{5pt}  \subfloat[$t = \frac{3 \pi}{4}$]{\includegraphics[width=0.3\textwidth]{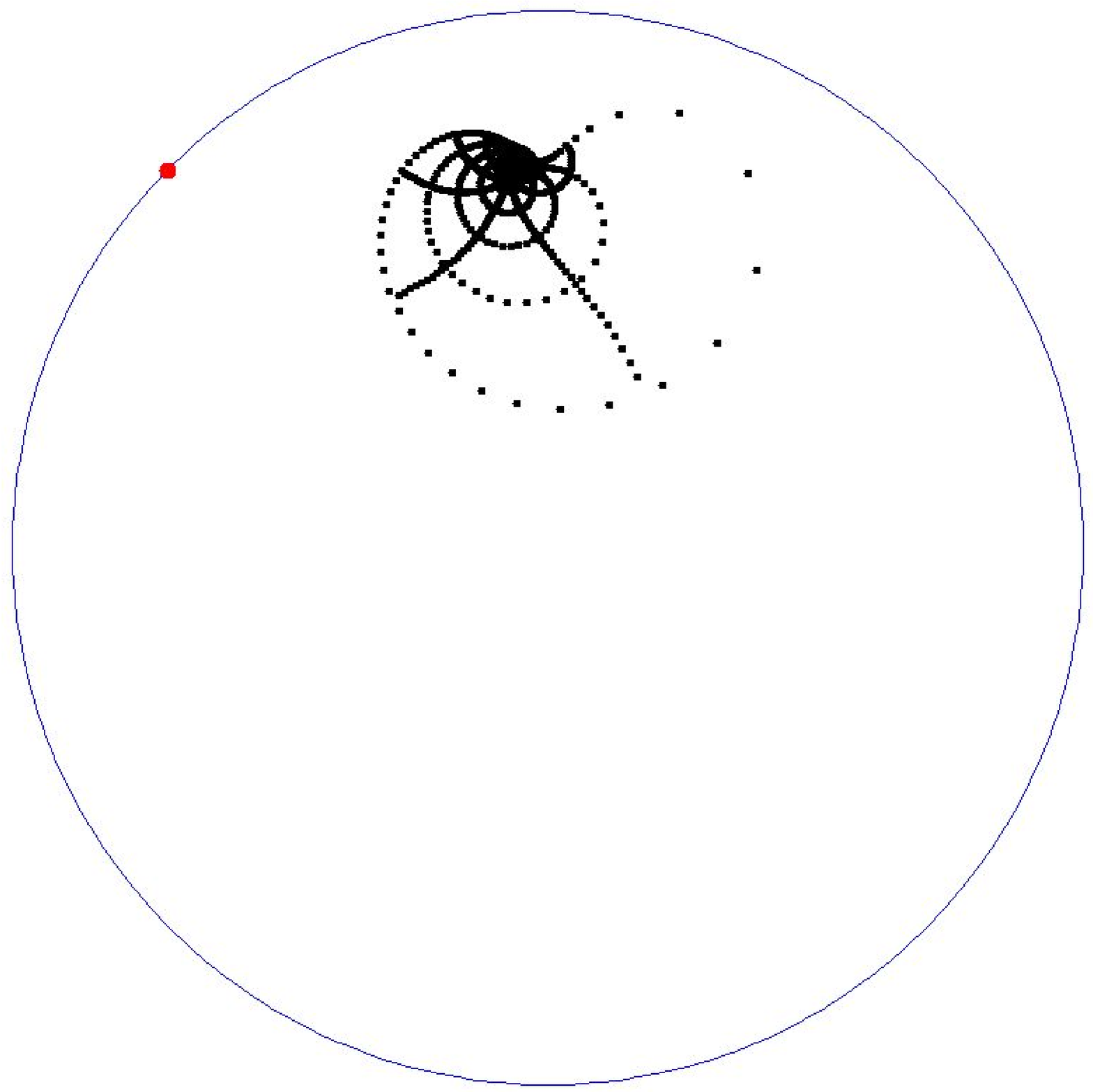}}
  \caption{Evolution of $\phi_t(\mathbb{D})$ for $p(z,t) = \exp(\pi z / 2 \tau )$, $k=1$ (the  point in $\mathbb T$  represents $\tau(t)$). }
\label{fig:expevolution}
\end{figure}

A Java applet demonstrating examples with different Herglotz functions and different values of $k$ is available at \url{http://folk.uib.no/giv023/loewnershapes}.

\subsection{Families of automorphisms} Let us describe the class of Herglotz functions generating families of automorphisms of the unit disk, using several basic results of  generation theory of semigroups of holomorphic self-maps of the unit disk. A summary of these results can be found, e.~g., in \cite[Chapter 2]{ShoikhetLinearization}.

\begin{proposition}
 Let the semigroup $\{\psi_t(z)\}_{t\geq 0}$ defined by \eqref{eq:autonomLoewner}   for some $k=k_0\in \mathbb{R}$ be a family of automorphisms of $\mathbb{D}$. Then $\{\psi_t(z)\}_{t\geq 0}$ is a family of automorphisms  $\mathbb{D}$ for all $k\in \mathbb{R}.$
\end{proposition}
\begin{proof}
Note that since $\{\psi_t(z)\}_{t\geq 0}$ consists of automorphisms    $\mathbb{D}$ for $k=k_0,$ the vector field $f_1(z) = (1-z)^2 \tilde{p}(z) - i k_0 z$ is complete  (i.~e., the solution of the corresponding differential equation exists and is unique for all $-\infty < t < +\infty$). Also note that for a $k\in \mathbb{R}$ the vector field $f_2(z) = i (k_0 - k )z$ is complete. Now, the vector field $f_3(z) = (1-z)^2 \tilde{p}(z) - i k z$ is complete as a sum of two complete vector fields: $f_3 = f_1 + f_2,$ and thus, the generated semigroup $\psi_t(z)$ is a family of automorphisms.
\end{proof}

\begin{proposition}
 The family $\{\psi_t(z)\}_{t\geq 0}$ consists of automorphisms if and only if the corresponding Herglotz function is of the form $\tilde{p}(z) = A \, \frac{1+z}{1-z} + B\, i,$ where $A \geq 0,$ $B \in \mathbb{R}.$
\end{proposition}
\begin{proof}
As the previous proposition suggests, it is sufficient to prove the proposition for $k=0.$ Let $f(z)$ be the generator of $\{\psi_t(z)\}_{t\geq 0}$ for $k=0,$ i.~e.,  $f(z) = (1-z)^2 \tilde{p}(z).$

Since $f(z)$ is a complete vector field, it must be a polynomial of the form
\[
 f(z) = -\bar{a} z^2 + i b z +a, \quad a\in \mathbb{C}, \quad b \in \mathbb{R},
\]
(see, e.~g., \cite[Theorem 2.6]{ShoikhetLinearization}). Since 1 is a fixed point of the semigroup, it follows  from the relation $f(1) =0,$ that $b = - 2 \im a.$ Hence, $f(z)$ is a polynomial of the form
\[
 f(z) = (-A + B i)\,z^2 - 2 B i z + (A + B i),
\]
which is precisely $(1-z)^2 \tilde{p}(z)$ with $\tilde{p}(z) = A \frac{1+ z}{1-z} + B i.$  The function $\tilde{p}(z)$ has nonnegative real part if and only if $A \geq 0.$ (The case $A<0$ corresponds to the Berkson-Porta decomposition of $f(z)$ with the attracting point $-\frac{A+ B i}{A - B i}$ and the Herglotz function $\tilde{p}(z) = \frac{(z-1) \left(A^2+B^2\right)}{z (A-i B)+A+i B}$).
\end{proof}

Since $\tilde{p}(0) = A + B i,$ we will also write the Herglotz functions generating families of automorphisms as  $\tilde{p}(z) = \frac{1+z}{1-z} \re \tilde{p}(0) + i \im \tilde{p}(0).$

The differential equation for the semigroup generated by $\tilde{p}(z) = A \, \frac{1+z}{1-z} +  B\, i $  is written as
\begin{equation}
 \label{eq:automorphpsi}
 \frac{d \psi_t}{(-A+i B) \, \psi _t^2+(-2 i B-i k) \,\psi _t+A+i B}  = dt.
\end{equation}

Denote by D the discriminant of the polynomial in the denominator of the left-hand side: $D = 4 A^2-4 B k-k^2.$ 
\[
 \begin{cases}
  D > 0, \textrm{ if } 2\, \left(-\im \tilde{p}(0)  - |\tilde{p}(0)| \right) < k < 2\,(-\im \tilde{p}(0) + |\tilde{p}(0)|), \\
D = 0, \textrm{ if } k = 2\,(-\im \tilde{p}(0) \pm |\tilde{p}(0)|), \\
D < 0, \textrm{  if }  k < 2\, \left(-\im \tilde{p}(0)  - |\tilde{p}(0)| \right),  \textrm{ or } k > 2\,(-\im \tilde{p}(0) + |\tilde{p}(0)|).
 \end{cases}
\]

Depending on the sign of $D$, the nature of the semigroup $\psi_t$ changes qualitatively. One can check that 
\[
\psi_t(z) \textrm{ is }\begin{cases}
  \textrm{hyperbolic, if } 2\, \left(-\im \tilde{p}(0)  - |\tilde{p}(0)| \right) < k < 2\,(-\im \tilde{p}(0) + |\tilde{p}(0)|), \\
 \textrm{parabolic, if } k = 2\,(-\im \tilde{p}(0) \pm |\tilde{p}(0)|), \\
\textrm{elliptic,  if }  k < 2\, \left(-\im \tilde{p}(0)  - |\tilde{p}(0)| \right),  \textrm{ or } k > 2\,(-\im \tilde{p}(0) + |\tilde{p}(0)|).
\end{cases}
\]

In the elliptic case  \eqref{eq:automorphpsi} leads to the following equation for $\psi_t(z)$
\begin{equation}
\label{eq:psilog}
\frac{\sqrt{-D} + 2 i A \psi_t(z) + 2 B (1 - \psi_t(z)) - k}{\sqrt{-D} - 2 i A \psi_t(z) - 2 B (1 - \psi_t(z)) + k} \cdot \frac{\sqrt{-D} - 2 i A z - 2 B (1 - z) + k}{\sqrt{-D} + 2 i A z + 2 B (1 - z) - k} = e^{- i \sqrt{-D} t}.
\end{equation}

A necessary and sufficient condition for the curve $\psi_t(z)$ to be closed is the existence of a value $t_0$ of $t,$ such that $\psi_{t_0} (z) = z$ and $\dot{\psi}_{t_0} = \dot{\psi}_0(z),$ i.~e.,  both the position and velocity coincide with the position and velocity at the initial moment $t=0$.  It follows from \eqref{eq:psilog} that $\psi_t(z)$ returns to the point $z$ for $t = \frac{2 \pi m}{\sqrt{-D}},$ $m \in \mathbb{Z}.$ From the original equation it follows that the velocity of a point of the trajectory $\psi_t(z)$ passing through the point $z$ is $\frac{(e^{ikt} - z)^2}{e^{ikt}}\, \tilde{p}(z /e^{ikt}),$ i.~e., it is $\frac{2\pi n}{k}$-periodic. Thus, it is necessary and sufficient that there exist integers $m,n \in \mathbb{Z},$ such that $\frac{2 \pi m}{\sqrt{-D}} = \frac{2\pi n}{k},$ and we arrive at the following proposition.

\begin{proposition}
\label{prop:closed}
 Let $\{\psi_t(z)\}_{t\geq 0}$ be an elliptic semigroup of automorphisms, i.~e., $\psi_t(z)$ solves the initial value problem,
\[
\begin{cases}
 \frac{d}{dt}\psi_t(z) = (\psi_t(z)-1)^2 (A \frac{1+\psi_t(z)}{1-\psi_t(z)} + B i) - ik\psi_t(z),\\
\psi_0(z)=z,
\end{cases}
\]
for $k < 2\, \left(-B  - \sqrt{A^2+B^2} \right)$,  or $k > 2\,(-B + \sqrt{A^2+B^2})$. 

Then the curves $\psi_t(z)$ are  closed if and only if
\[
 \frac{k}{\sqrt{ - 4 A^2 + 4B k + k^2}} \in \mathbb{Q}.
\]

\end{proposition}

Figures \ref{fig:Cayley} and \ref{fig:pi05} illustrate Proposition \ref{prop:closed}.
\begin{figure}
\center
\epsfig{file=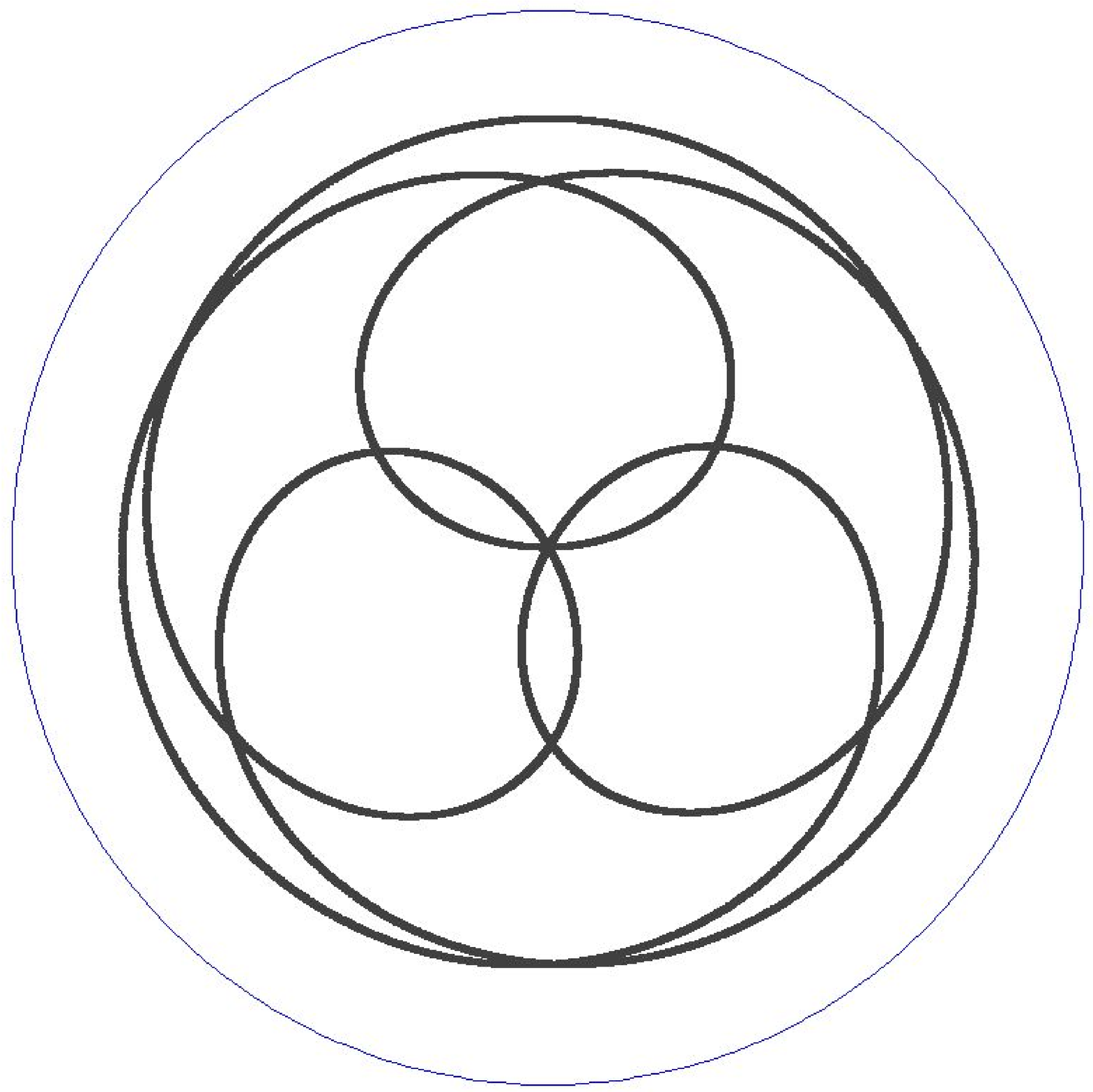,height=6cm}
\caption{The closed curve $\phi_t(0)$ for $p=\frac{\tau+z}{\tau-z}$, $k=2.5.$}
 \label{fig:Cayley}
\end{figure}

\begin{figure}[h]
\center
\epsfig{file=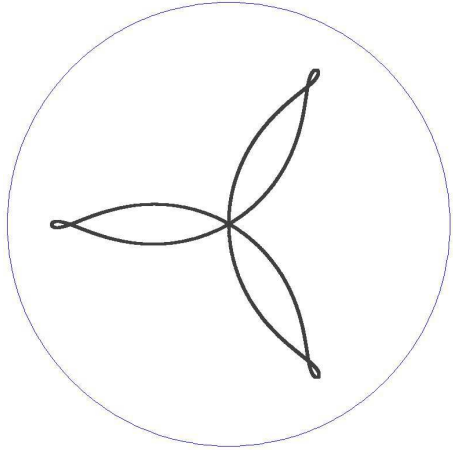,height=6cm}
 \caption{The closed curve $\phi_t(0)$ for $p=i$, $k=0.5.$}
 \label{fig:pi05}
\end{figure}

\subsection{Sufficient conditions for ellipticity}
\label{subseq:suffellipticity}
We answered the question about the values of $k$ for which the semigroup $\{\psi_t(z)\}_{t\geq 0}$ is elliptic, hyperbolic or parabolic completely in the case of automorphisms. Let us consider the general case of semigroups of holomorphic self-maps of $\mathbb{D}$

In the case of $k=0$ the semigroup $\{\psi_t(z)\}_{t\geq 0}$ is either hyperbolic (if $\angle \lim_{z\to 1} (1-z) \tilde{p}(z)  < 0$) or parabolic (if $\angle \lim_{z\to 1} (1-z) \tilde{p}(z) = 0$). 

In the rest of this section we assume $k \neq 0.$ Since the semigroup $\{\psi_t(z)\}_{t\geq 0}$ satisfies the equation
\[
  \frac{d}{dt}\psi_t(z) = (\psi_t(z)-1)^2 \tilde{p}(\psi_t(z)) - ik\psi_t(z),
\]
it is necessary and sufficient for ellipticity of $\psi_t$ that the equation
\begin{equation}
\label{eq:generatorzero}
 (1-z)^2 \tilde{p}(z) - ik z = 0 
\end{equation}
has a solution $z_0$ inside $\mathbb{D}.$ 

Equation \eqref{eq:generatorzero} can be rewritten as
\begin{equation}
\label{eq:generatorzerokoebe}
 \tilde{p}(z) = \koebe_k ( z),
\end{equation}
where the function 
\[
 \koebe_k(z) = \frac{ikz}{(1-z)^2}
\]
is the rotated and rescaled Koebe function $\frac{z}{(1-z)^2}.$ The image of the unit circle $\mathbb{T}$ under $\koebe_k$ is the ray along the imaginary axis going from $ik$ to $\infty.$ The inverse function is given by 
\[
 \koebe_k^{-1}(z) = \frac{\sqrt{\frac{4}{ik}z + 1}-1}{\sqrt{\frac{4}{ik} z + 1}+1}.
\]

If we rewrite \eqref{eq:generatorzerokoebe} as 
\[
 \koebe_k^{-1} (\tilde{p}(z)) = z,
\]
then the question about  zeros of the infinitesimal generator of $\psi_t$ reduces to the question about fixed points of the map $F(z) = \koebe_k^{-1} (\tilde{p}(z))$ in $\overline{\mathbb{D}}.$

\begin{proposition}
 Let $\re p(z) > \epsilon >0 $ for all $z\in \mathbb{D}.$ Then $\psi_t(z)$ is elliptic for all $k \neq 0.$
\end{proposition}
\begin{proof}
Solutions to the equation $K_k^{-1}(\tilde{p}(z))=z$ in the closed unit disk $\overline{\mathbb{D}}$ must lie in $K_k^{-1}(\tilde{p}(\mathbb{\overline{D}})).$ The intersection $K_k^{-1}(\tilde{p}(\overline{\mathbb{D}})) \cap \mathbb{T}$ may contain at most one point $z=1,$ and $1$ cannot be a solution of \eqref{eq:generatorzero} for $k\neq 0.$ Hence, the infinitesimal generator must have a zero  inside $\mathbb{D},$ and the semigroup $\{\psi_t(z)\}_{t\geq 0}$ is elliptic.
\end{proof}

The same argument leads to a more general statement.

\begin{proposition}
 Let for a given $k$ the Euclidean distance $\dist(\tilde{p}(\mathbb{\mathbb{T}}), \koebe_k (\mathbb{T})) > \epsilon> 0.$ Then $\psi_t(z)$ is elliptic for this value of $k$.
\end{proposition}

Next proposition states that,  in general, large values of $k$  correspond to elliptic semigroups.

\begin{proposition}
For each Herglotz function $\tilde{p}(z)$ there exists $k_0,$ such that for all $|k|>k_0$ the semigroup $\{\psi_t(z)\}_{t\geq 0}$ is elliptic. Moreover, the Denjoy-Wolff point of $\psi_t(z)$ can approach $0$ arbitrarily close, provided that $k_0$ is large enough.
\end{proposition}
\begin{proof}
 Let $\overline{\mathbb{D}}_r = \{z : |z| \leq r\}.$ It is easy to see that for $k$ large enough the map 
\begin{equation}
\label{eq:Fdef}
F(z)=  \koebe_k^{-1} (\tilde{p}(z))= \frac{\sqrt{\frac{4}{ik}\tilde{p}(z) + 1}-1}{\sqrt{\frac{4}{ik} \tilde{p}(z) + 1}+1} 
\end{equation}
is a strict contraction. Indeed, the set $\tilde{p}(\overline{\mathbb{D}}_r)$ is bounded, so for  large $k,$ the set $\frac{4}{ik} \tilde{p}(\overline{\mathbb{D}}_r)$ is contained in an arbitrarily small neighborhood of  $0,$ the sets $\frac{4}{ik} \tilde{p}(\overline{\mathbb{D}}_r) + 1$ and $\sqrt{\frac{4}{ik} \tilde{p}(\overline{\mathbb{D}}_r) + 1}$ are contained in arbitrarily small neighborhoods of $1$, and $F(\overline{\mathbb{D}}_r)$ is contained in an arbitrarily small disk centered at the origin. Thus, $F$ is a strict contraction and has a unique fixed point $z_0 =F(z_0) ,$ and the fixed point $z_0$ lies in a small disk centered at the origin.
\end{proof}

\begin{proposition}
 Suppose that for a given $k \neq 0$  the semigroup $\psi_t(z)$ defined by \eqref{eq:autonomLoewner} is elliptic, and let $z_0$ be its Denjoy-Wolff point. If $k >0,$ then $\im z_0 \leq 0,$ and if $k <0,$ then $\im z_0 \geq 0.$ Moreover, the equality $\im z_0 = 0$  is possible, if and only if, $\tilde{p}(z) \equiv B i,$ $B \in \mathbb{R}.$
\end{proposition}
\begin{proof}
 Since the Denjoy-Wolff point of $\psi_t(z)$  is the fixed point of $F(z)$ defined by \eqref{eq:Fdef}, then it must lie in  the domain $F(\mathbb{D}).$ For $\tilde{p}(z) \equiv B i,$ $F(\mathbb{D}) = \{1 + \frac{1}{2B} \left(k - \sqrt{k^2 + 4 B}\right)\} \subset \mathbb{R}.$ Otherwise,  $\im F(\mathbb{D}) > 0,$ and we obtain the result.
\end{proof}

Of course, it might happen that the semigroup $\psi_t$ has boundary fixed points, even though it is elliptic. For example, let $\tilde{p}(z) = \frac{1-z}{1+z}.$ Then the function $F(z) = \koebe_k^{-1} (\tilde{p}(z))$ has fixed points on $\mathbb{T}$ for all $k.$ Indeed, the equation
\[
\frac{1-e^{i\theta}}{1+ e^{i\theta}} = \frac{ik e^{i\theta}}{(e^{i\theta} -1)^2},
\]
which can be rewritten as
\[
 \tan \frac{\theta}{2} =\frac12 \frac{k }{\cos \theta - 1},
\]
has solutions for any $k$.

\section{Stochastic case}
\label{sec:stochastic}
In this section we consider the generalized L\"owner evolution driven by a Brownian particle on the unit circle. In other words, we study the following initial value problem.

\begin{equation}
\label{eq:randomloewner}
\begin{cases}
 \frac{d}{dt} \phi_t (z,\omega) = \frac{(\tau(t,\omega)  - \phi_t(z,\omega))^2}{\tau(t,\omega)} p(\phi_t (z,\omega),t,\omega), \\
\phi_0(z,\omega) = z,
\end{cases}
 t \geq 0, \, z \in\mathbb{D}, \, \omega \in \Omega.
\end{equation}

Here, $\Omega$ is the sample space in the standard filtered probability space $(\Omega,\mathcal F, (\mathcal F_t), P)$ of Brownian motion, $\tau(t,\omega) = \exp (ikB_t(\omega))$, where $k \in \mathbb{R}$, and $B_t(\omega)$ denotes the standard 1-dimensional Brownian motion. The function $p(z,t,\omega)$ is a Herglotz function for each fixed $\omega\in \Omega.$ In order for $\phi_t(z,\omega)$ to be an It\^o process adapted to the Brownian filtration, we require that the function $p(z,t,\omega)$ is adapted to the Brownian filtration for each $z\in\mathbb{D}.$

For each fixed $\omega\in \Omega,$ equation \eqref{eq:randomloewner} may be considered as a deterministic generalized L\"owner equation with the Berkson-Porta data $(\tau(\cdot,\omega),p(\cdot, \cdot, \omega)).$ In particular, the solution $\phi_t(z,\omega)$ exists, is unique for each $t>0$ and $\omega\in\Omega,$ and moreover, is a family of holomorphic self-maps of the unit disk.

The equation in (\ref{eq:randomloewner}) is an example of a so-called \emph{random differential equation} (see, for instance, \cite{soong73}). Since for each fixed $\omega \in\Omega$ it may be regarded as an ordinary differential equation, the sample paths $t \mapsto \phi_t(z,\omega)$ have continuous first derivatives for almost all $\omega$ (Figure \ref{fig:randomloewner}). This is entirely different from the case of It\^{o} and Stratonovich \emph{stochastic differential equations}, where sample paths are nondifferentiable and tools of stochastic calculus are needed to solve those equations.

\begin{figure}[h]
 \centering
 \includegraphics[width=0.4 \textwidth]{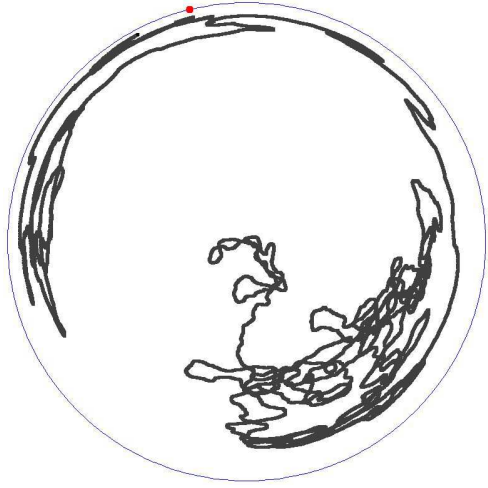}
 % randomcayley5_30.eps: 780x780 pixel, 72dpi, 27.52x27.52 cm, bb=0 0 780 780
 \caption{A sample path of $\phi_t(0,\omega)$ for $p(z,t) = \frac{\tau(t) + z }{ \tau(t)-z }$, $\tau(t) = e^{ikB_t},$ $k = 5$, $t \in[0,30]$.}
 \label{fig:randomloewner}
\end{figure}

\subsection{An explicitly solvable example}

Let $p(z,t,\omega) = \frac{\tau(t,\omega)}{\tau(t,\omega) - z}= \frac{e^{ikB_t(\omega)}}{e^{ikB_t(\omega)} - z},$ that is, the stochastic counterpart of the Herglotz function in Example 1 from the previous section. It makes  equation  \eqref{eq:randomloewner} linear:

\[
\frac{d}{dt} \phi_t (z,\omega) = e^{ikB_t(\omega)} - \phi_t(z,\omega),
\]
and a well-known formula from the theory of ordinary differential equation yields
\[
 \phi_t(z,\omega) = e^{-t} \left(z + \int_0^t e^{s} e^{ikB_s(\omega)}ds\right).
\]

Taking into account the fact that $\mathsf{E} e^{ikB_t(\omega)} = e^{-\frac12 t k^2},$ we can also write the expression for the mean function  $ \mathsf{E}\phi_t(z,\omega)$

\begin{equation}
\label{eq:meanfunctionphi}
\mathsf{E}\phi_t(z,\omega) = \begin{cases}
                               e^{-t}(z+t), \quad \quad \quad \quad k^2 = 2, \\
			      e^{-t} z + \frac{e^{-tk^2/2}-e^{-t}}{1-k^2/2},\quad \textrm{otherwise.}
                              \end{cases}
\end{equation}

Thus, in this example all maps $\phi_t$ and $\mathsf{E} \phi_t$ are affine transformations (compositions of a scaling and a translation).

In general, solving the random differential equation (\ref{eq:randomloewner}) is much more complicated than solving its deterministic counterpart, mostly because of the fact that for almost all $\omega$ the function $t \mapsto \tau(t,\omega)$ is nowhere differentiable.

\subsection{Invariance property}
If we assume again, as we did in the deterministic case, that the Herglotz function has the form $p(z,t,\omega) = \tilde{p}(z/\tau(t,\omega)),$ then it turns out that the process $\phi_t(z,\omega)$ has an important invariance property, that were crucial in development of SLE.

Let $s>0$ and introduce the notation 
\[
 \tilde{\phi}_t(z) = \frac{\phi_{s+t}(z)}{\tau(s)}. 
\]
Then $\tilde{\phi}_t(z)$ is the solution to the initial-value problem
\[
 \begin{cases}
  \frac{d}{dt} \tilde{\phi}_t (z,\omega) = \frac{\left(\tilde{\tau}(t,\omega)  - \tilde{\phi}_t(z,\omega)\right)^2}{\tilde{\tau}(t,\omega)} \tilde{p}\left(\tilde{\phi}_t (z,\omega)/\tilde{\tau}(t)\right), \\
\tilde{\phi}_0(z,\omega) =\phi_s(z,\omega)/\tilde{\tau}(s), 
 \end{cases}
\]
where $\tilde{\tau}(t) = \tau(s+t)/ \tau(s) = e^{ik(B_{s+t}-B_s)}$ is again a Brownian motion on $\mathbb{T}$ (because $\tilde{B}_t = B_{s+t}-B_s$ is a standard Brownian motion). In other words, the conditional distribution of $\tilde{\phi}_{t}$ given $\phi_{r},$ $r\in[0,s] $ is the same as the distribution of $\phi_t.$

\subsection{Stochastic change of variables}

By the complex It\^o formula (Appendix \ref{subseq:ComplexIto}), the process $\frac{1}{\tau(t,\omega)} = e^{-ik B_t}$ satisfies the equation

\[
 d e^{-ikB_t} = -ik e^{-i k B_t} dB_t - \frac{k^2}{2} e^{- ikB_t}dt.
\]

Let us denote $\frac{\phi_t(z,\omega)}{\tau(t,\omega)}$ by $\Psi_t(z,\omega).$ Applying the integration by parts formula to $\Psi_t,$ we arrive at the following initial value problem for the It\^o stochastic differential equation
%If we assume again that the Herglotz function has the form $p(z,t,\omega) = \tilde{p}(z/\tau(t,\omega))$ and, just as we did in the deterministic case,  introduce the new variable $\Psi_t(z,\omega) %= \frac{\phi_t(z,\omega)}{\tau(t,\omega)},$ then, using the It\^{o} formula, we arrive at the following  problem for an It\^o stochastic differential equation
\begin{equation}
\begin{cases}
\label{eq:stochasticloewner}
  d\Psi_t =  - ik \Psi_t dB_t  + \left( - \frac{k^2}{2}\Psi_t + (\Psi_t -1)^2 p(\Psi_t e^{ikB_t(\omega)}, t, \omega)\right) dt, \\
\Psi_0(z) = z.
\end{cases}
\end{equation}
\begin{comment}or, in Stratonovich form,

\[
\begin{cases}
  d\Psi_t =  - ik \Psi_t \circ dB_t  + (\Psi_t -1)^2  p(\Psi_t e^{ikB_t(\omega)}, t, \omega) dt,\\
\Psi_0(z) = z.
\end{cases}
\]
\end{comment}
A numerical solution to this equation for a specific choice of $p(z,t, \omega)$ is shown in Figure~\ref{fig:stochasticpi}.
\begin{figure}[h]
 \centering
 \includegraphics[width=0.4 \textwidth]{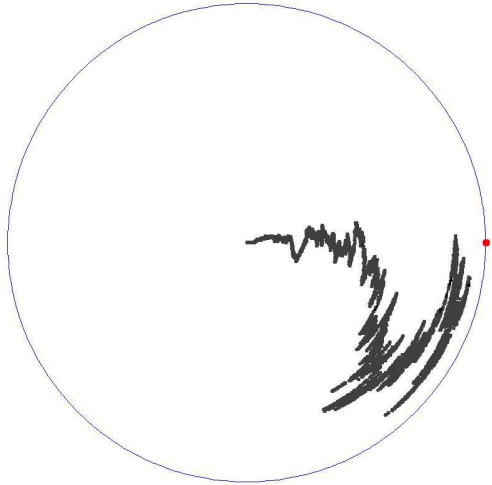}
 % stochasticcayley.eps: 780x780 pixel, 72dpi, 27.52x27.52 cm, bb=0 0 780 780
 \caption{A sample path of $\Psi_t(0,\omega)$ for $p(z,t) \equiv i,$ $k=1$ and $t\in[0,2].$}
 \label{fig:stochasticpi}
\end{figure}

Analyzing the process  $\frac{\phi_t(z,\omega)}{\tau(t,\omega)}$ instead of the original process $\phi_t(z,\omega)$ is in many ways similar to one of the approaches used in SLE theory. Recall that many properties of SLE$_k,$ that is, of the solution to the problem
\[
\begin{cases}
 \frac{d}{dt}g_t(z) =\frac{2}{g_t (z)- \sqrt{k} B_t},\\
g_0(z) = z,
\end{cases}
t\geq 0, \quad  \im z > 0,
 \]
can be proved using the fact that $g_t(z) - \sqrt{k} B_t$ is a so-called Bessel process.

The image domains $\Psi_t(\mathbb{D},\omega)$ differ from $\phi_t(\mathbb{D},\omega)$ only by rotation. Due to the fact that $|\Psi_t(z,\omega)|=|\phi_t(z,\omega)|$, if we compare the processes $\phi_t(0,\omega)$ and $\Psi_t(0,\omega),$ we note that their first hit times of the circle $\mathbb{T}_r$ with radius $r<1$ coincide, i.~e.,
\[
\inf \{t \geq 0, |\phi_t(0,\omega)| = r \} = \inf \{t \geq 0, | \Psi_t(0,\omega)| = r \}.                                                                                                                                                                                                                                                                                                                                                                                                                                                                                                                                                                                                                                                                                                                                                                                                                                                                                                                                                                                                                                                                                                                                 \]
In other words, the answers to probabilistic questions about the expected time of hitting the circle $\mathbb{T}_r,$ the probability of exit from the disk $\mathbb{D}_r = \{z : |z| < r\},$ etc. are the same for $\phi_t(0,\omega)$ and $\Psi_t(0,\omega).$

\subsection{$\Psi_t$ as a diffusion. Associated semigroup of operators}
If the Herglotz function has the form $p(z,t,\omega) = \tilde{p}(z/\tau(t,\omega))$, then the equation \eqref{eq:stochasticloewner} becomes
\begin{equation}
\begin{cases}
\label{eq:loewnerdiffusion}
  d\Psi_t =  - ik \Psi_t dB_t  + \left( - \frac{k^2}{2}\Psi_t + (\Psi_t -1)^2 \tilde{p}(\Psi_t)\right) dt, \\
\Psi_0(z) = z,
\end{cases}
\end{equation}
and may be regarded as an equation of a 2-dimensional time-homogeneous real diffusion written in complex form. This implies, in particular, that $\Psi_t$ is a time-homogeneous strong Markov process. By construction, $\Psi_t(z)$ always stays in the unit disk.

Let $B_b(\mathbb{D})$ denote the Banach space of bounded complex-valued Borel functions $f: \mathbb{D} \to \mathbb{C}$ with  supremum norm $\|f\| = \sup_{z\in \mathbb{D}} |f(z)|$. To a time-homogeneous Markov process $\Psi_t$ taking values in $\mathbb{D}$ one can associate a family of operators $T_t$ acting on $B_b(\mathbb{D})$ whose action is defined by the formula
\begin{equation}
 (T_t f)(z) = \mathsf{E} f (\Psi_t(z))
\end{equation}
(see \cite{Applebaum,Dynkin}).

When $t=0,$ the operator $T_t$ is the identity map. For all $t\geq 0,$ the operators $T_t$ are linear by linearity of expectation, and $\|T_t\|\leq 1,$ because
\[
\|T_t f \| = \sup_{z\in \mathbb{D}} |\mathsf{E} f(\Psi_t(z))| \leq \sup_{z\in \mathbb{D}} \mathsf{E} |f(\Psi_t(z))| \leq \sup_{z\in \mathbb{D}} |f(z)| = \|f\|.
\]

Since $\Psi_t$ is a time-homogeneous Markov process, the operators $T_t$ form a semigroup \cite{Applebaum,oksendal2003}:

\begin{equation}
\label{eq:Tsemigroup}
 T_{s+t} = T_s T_t.
\end{equation}

The operators $T_t$ preserve the class of bounded continuous functions and the Banach space of bounded analytic functions $H^{\infty}(\mathbb{D}),$ as stated in the two propositions below.

\begin{proposition}[{cf.\cite[proof of Theorem 6.7.2]{Applebaum}}] If $f:\mathbb{D} \to \mathbb{C}$ is a bounded  continuous function, then $T_t f$ is bounded and continuous for every $t\geq 0$.
\end{proposition}
\begin{proof}
 It has been shown above that $T_t$ is a bounded operator. It remains to show continuity. 

Let  $\{z_n\} _{n=1}^\infty \subset \mathbb{D}$ and $z_n \to z_0 \in \mathbb{D}.$  Since $f(z)$ and $\Psi_t(z,\omega)$ are continuous, we have that  $\lim_{n \to \infty} f(\Psi_t(z_n,\omega)) = f(\Psi_t(z_0,\omega))$ for every $\omega\in\Omega$. Since $f$ is bounded we also have that $|f(\Psi_t(z_n,\omega))| \leq \|f\| \in L^1(\Omega, dP).$ Hence, by the dominated convergence theorem, 
\[\lim_{n\to\infty} (T_t f)(z_n)= \lim_{n\to\infty} \mathsf{E} f(\Psi_t(z_n,\omega)) = \mathsf{E} f(\Psi_t(z,\omega)) = (T_t f)(z_0)\]
\end{proof}

\begin{proposition}
If $f\in H^{\infty}(\mathbb{D})$ then $T_t f \in H^\infty (z)$ for every $t\geq 0.$
\end{proposition}
\begin{proof}
 Let $f:\mathbb{D} \to \mathbb{C}$ be a bounded analytic function. $T_t f$ is continuous, and thus, by Morera's theorem, it suffices to show that  $\int_\Delta \mathsf{E} f(\Psi_t(z),\omega)dz = 0$ for every triangle $\Delta \subset \mathbb{D}$. But this follows directly from Fubini's theorem. Indeed, let $\gamma:[0,1] \to \mathbb{D}$ be a parameterization of the triangle $\Delta.$ Then, since
\[
\mathsf{E} \int_0^1 |f(\Psi_t(\gamma(t),\omega))\gamma'(t)|dt \leq \|f\| \cdot l(\gamma) <\infty,
\]
we conclude that

\[
 \int_\Delta \mathsf{E} f(\Psi_t(z,\omega)) dz = \mathsf{E} \int_\Delta f(\Psi_t(z,\omega)) dz = 0.
\]

\end{proof}

\subsection{Infinitesimal generator of $\Psi_t$}
The operator $A$ defined by the expression
\[
 Af(z) = \lim_{t \downarrow 0} \frac{(T_t f)(z) - f(z)}{t}, \quad z \in \mathbb{D},
\]
acting on a suitable class of functions $f:\mathbb{C} \to \mathbb{C}$ is called the \emph{infinitesimal generator} of the It\^o diffusion $\Psi_t$. The set of functions for which $Af$ is defined is denoted by $\mathcal{D}_A,$ i.~e. $f\in \mathcal{D}_A$ if the limit above exists for all $z\in \mathbb{D}.$

Following the lines of the standard proof for the case of real diffusions and applying the complex It\^o formula (see Appendix \ref{subseq:ComplexIto})  one shows that $C^2_0(\mathbb{C}) \subset \mathcal{D}_A,$ and that the infinitesimal generator in this case is given by the formula
\begin{multline*}
  A = \left(-\frac{k^2}{2}z + (z-1)^2 \tilde{p}(z)\right) \frac{\partial}{\partial z} -\frac{1}{2}k^2z^2 \frac{\partial^2}{\partial z^2} \\
+\left(-\frac{k^2}{2}\bar{z} + (\bar{z}-1)^2 \overline{\tilde{p}(z)}\right) \frac{\partial}{\partial \bar{z}} -\frac{1}{2}k^2\bar{z}^2 \frac{\partial^2}{\partial \bar{z}^2} + k^2 |z|^2 \frac{\partial^2 f}{\partial z \partial \bar{z}}.
\end{multline*}

Now, let $f$ be analytic in $\mathbb{D}$. Then, by It\^o's formula for analytic functions,

\begin{multline*}
 \mathsf{E} f(\Psi_t(z)) - f(z) \\= \mathsf{E} \int_0^t \left[ \left(-\frac{k^2}{2} \Psi_s(z) + (\Psi_s(z) - 1)^2 \tilde{p} \left(\Psi_t(z)\right )\right) \frac{\partial f}{\partial z} (\Psi_t(z)) - \frac{k^2}{2} \Psi_s^2(z) \,\frac{\partial^2 f}{\partial z^2}(\Psi_s(z))\right]dt,
\end{multline*}
and thus,
\begin{equation}
\label{eq:generator}
 A = \left(-\frac{k^2}{2}z + (z-1)^2 \tilde{p}(z)\right) \frac{\partial}{\partial z} -\frac{1}{2}k^2z^2 \frac{\partial^2}{\partial z^2}.
\end{equation}

%Let $f\in C_0^2$ be annihilated by the infinitesimal generator, i.~e.,  $Af =0.$ Then, by the It\^o formula, the process $f(\Psi_t(z))$ is given by an It\^o integral, and hence, is a local martingale.

Let $\tilde{p}(z) = \sum_{n=0}^\infty a_n z^n$ be the Taylor expansion of $\tilde{p}(z)$ in $\mathbb{D}.$ Set
\[
 \mathcal{L}_n = - z^{n+1} \frac{\partial}{\partial z}, \quad  n = 0, \pm 1, \pm 2, \ldots
\]
Observe that $\mathcal{L}_n$ give the Virasoro generators in a representation of the Virasoro algebra over the space of complex Laurent polynomials. Interesting connections between highest-weight representations of the Virasoro algebra and properties of SLE were studied, e.~g., in \cite{BB, FriedrichWerner}

Using the formula $z^2 \frac{\partial^2}{\partial z^2} = \mathcal{L}_0^2 + \mathcal{L}_0$,  we can rewrite (\ref{eq:generator}) as
\begin{equation}
\label{eq:generatorvirasoro}
 A =- \sum_{n=-1}^\infty (a_{n+1} - 2 a_n + a_{n-1}) \mathcal{L}_n  -\frac{k^2}{2} \mathcal{L}_0^2, \quad a_{-1}=a_{-2}=0.
\end{equation}
Note that operators $\mathcal{L}_n$ with $n < -1$ never appear in this expression for $A$.

We can now write  the generators for particular functions $\tilde{p}(z)$ considered in the previous section.

For $\tilde{p}(z) = \frac{1}{1-z},$ \[A = - \mathcal{L}_{-1} + \mathcal{L}_0 - \frac{k^2}{2}  \mathcal{L}_0^2.\]

For $\tilde{p}(z) = \frac{1+z}{1-z},$  \[A = - \mathcal{L}_{-1}+\mathcal{L}_1 - \frac{k^2}{2} \mathcal{L}_0^2.\]

For $\tilde{p}(z) = i,$ \[A= - i \mathcal{L}_{-1} + 2i \mathcal{L}_{0} - i \mathcal{L}_1 - \frac{k^2}{2}\mathcal{L}_0^2.\]

Since the first term in \eqref{eq:generatorvirasoro} depends linearly on $\tilde{p}(z)$ and the second term does not depend on it, it is easy to obtain the generator for $\tilde{p}(z) = A\, \frac{1+z}{1-z} + B i$:
\[
 A = - (A + B i) \,\mathcal{L}_{-1} + 2 i B \mathcal{L}_0 + (A - B i ) \,\mathcal{L}_1 - \frac{k^2}{2} \mathcal{L}_0^2.
\]

\subsection{Analytic properties of the family $T_t f$. Kolmogorov's backward equation}
Let $f$ be a bounded analytic function, $f\in H^\infty(\mathbb{D}).$ From Dynkin's formula
\[
 E[f(\Psi_t(z))] = f(z) + E\left[\int_0^t (Af)(\Psi_s(z))ds\right]
\]
it follows that the function $u(t,z) = (T_tf)(z)$ is differentiable with respect to $t$,
and
\[
 \frac{\partial }{\partial t}u(t,z) = \mathsf{E} \left[(Af)(\Psi_t(z))\right].
\]

The family of functions $(T_t f)(z)$ is a uniformly bounded family of holomorphic functions in $\mathbb{D}.$ Hence, it is normal by Montel's theorem. Since this family is pointwise continuous, it is also continuous in the topology of local uniform convergence by Vitali's theorem.

The function $u(t,z) = (T_t f)(z)$ satisfies the following initial value problem for the Kolmogorov backward equation.

\[
\begin{cases}
 \frac{\partial u}{\partial t} = \left(-\frac{k^2}{2}z + (z-1)^2 \tilde{p}(z)\right) \frac{\partial u}{\partial z} - \frac12 k^2 z^2 \frac{\partial^2 u}{\partial z^2},\\
u(0,z)= f(z),
\end{cases}
 z\in \mathbb{D}.
 \]

\subsection{Relation to the deterministic case}
Let us denote by $u(t,z)$  the mean function of $\Psi_t(z),$ which is the result of acting by $T_t$ on the identity function: $u(t,z) = (T_t \id )(z) = \mathsf{E} \Psi_t(z).$ It follows from \eqref{eq:Tsemigroup} that
\[
 u(t+s, z) = \mathsf{E} u(0, \Psi_{t+s}(z)) = \mathsf{E} u(t, \Psi_s(z)) = \mathsf{E} u(s,\Psi_t(z)).
\]
This reveals a deep similarity between the deterministic and stochastic cases. While the deterministic family $\psi_t(z)$  of holomorphic self-maps of the disk forms a bona fide semigroup, the stochastic family $\Psi_t$ forms a semigroup only in average. 

We can also look at this connection from a different point of view: to the semigroup $\psi_t(z),$ as well as to the diffusion $\Psi_t,$ one can associate a family of operators $T_t,$ which in the deterministic case are defined trivially by $(T_t f)(z) =f(\psi_t(z))$.

\subsection{Zero points of the stochastic vector field}
Analogously to the deterministic case, we can regard (\ref{eq:stochasticloewner}) as an autonomous L\"owner equation
\[
 \frac{d}{dt}\Psi_t(z,\omega) = G_0(\Psi_t(z,\omega)),
\]
where the Herglotz vector field $G_0(z,\omega)$ is given by
\[
G_0(z,\omega) = -ikz W_t(\omega)  - \frac{k^2}{2}z+  (z -1)^2 \tilde{p}(z).
\]
Here, $W_t(\omega)$ denotes a generalized stochastic process  known as  \emph{white noise} (see \cite{oksendal2003} for details).

By the argument we used in subsection \ref{subseq:suffellipticity} it follows that for $k\neq 0$ there  exists a unique point $z_0 \in \mathbb{D},$ such that $-\frac{k^2}{2}z_0 + (z_0 -1)^2 \tilde{p}(z_0).$

By the definition of the infinitesimal generator of a diffusion, $z_0$ is the unique point in the unit disk having the property
\[
 \lim_{t\to 0} \mathsf{E} \frac{\Psi_t(z) - z}{t} =  -\frac{k^2}{2}z + (z -1)^2 \tilde{p}(z).
\]

\subsection{Hierarchy of differential equations for moments}

We use the term \emph{the $n$-th moment of the process $\Psi_t(z)$}  for the function $\mu_m (z) = T_t z^m= \mathsf{E}\Psi^m_t(z).$ The first moment of $\Psi_t$ is thus, the mean function $\mathsf{E} \Psi_t(z).$
 
By applying the It\^o formula to the function $f(z) = z^m$ and taking expectations we obtain
\[
 d \mu_m =- \frac{k^2 m^2}{2} \mu_m + m \mathsf{E} \left[ \Psi^{m-1}_t (\Psi_t - 1)^2\tilde{p}(\Psi_t) \right]dt.
\]

For a function $\tilde{p}(z) = \sum_{n=0}^\infty a_n z^n$ this leads to an infinite-dimensional system of differential equations for the functions $\mu_1, \mu_2, \ldots:$

\begin{equation}
\label{eq:hierarchy}
 \frac{d}{dt} \mu_m = a_0 m \mu_{m-1} + \left(a_1 - 2 a_0 - m \,\frac{k^2}{2} \right) m \mu_m + \sum_{n=1}^\infty(a_{n-1} - 2 a_n + a_{n+1}) m \mu_{m+n}.
\end{equation}

In the simplest case when $\tilde{p}(z) = \frac{1}{1-z},$ this reduces to 

\[
 \frac{d}{dt} \mu_m =  m \mu_{m-1} - m \left(m \frac{k^2}{2}+1\right) \mu_m,
\]
and for $m = 1$ can be easily solved explicitly:

\[
 \mu_1(z) = \frac{1 - e^{-(1 +k^2/2)\,t}}{1 + k^2/2} + z e^{-(1 + k^2/2)\,t}.
\]

As a by-product, we can find the expression for the covariance of the processes $\int_0^t e^s e^{ikB_s} ds$ and $e^{-ikB_s}$ simply by comparing the formula above with the formula \eqref{eq:meanfunctionphi} for $\mathsf{E} \phi_t(z):$

\[
 \mathsf{E} e^{-i k B_t} = e^{-\frac{k^2}{2}t},
\]

\[
 \mathsf{E} \int_0^t e^s e^{-ikB_s}ds  = \begin{cases}
                                                                        e^{-t} t, \quad \quad \quad \quad k^2 = 2, \\
			      \frac{e^{-tk^2/2}-e^{-t}}{1-k^2/2},\quad \textrm{otherwise,} 
                                         \end{cases}
\]

\[
 \mathsf{E}\left( \int_0^t e^s e^{-ikB_s}ds  \cdot e^{-i k B_t}\right) =\frac{1 - e^{-(1 +k^2/2)\,t}}{1 + k^2/2},
\]

\begin{multline*}
 \mathrm{Cov} \left(\int_0^t e^s e^{ikB_s}ds, e^{-ikB_t}\right) = \mathsf{E}  \left[\int_0^t  e^{ikB_s}\,ds \cdot e^{-ikB_t }\right] - \mathsf{E} \int_0^t  e^{ikB_s}\,ds \cdot \mathsf{E} e^{-ikB_t }
\\= \begin{cases} \frac{1}{3}-\frac{1}{3} e^{-3 t} (3 t+1), \quad \quad \quad \quad \quad \quad  k^2 =2, \\ 
 \frac{1 - e^{-(1 +k^2/2)t}}{1 + k^2/2}- \frac{e^{-tk^2/2}-e^{-t}}{1-k^2/2} e^{-\frac{1}{2}tk^2}, \quad \textrm{otherwise.}
    \end{cases}
\end{multline*}

Even though it is impossible  to solve  \eqref{eq:hierarchy} explicitly for other Herglotz functions $\tilde{p}(z)$, the system still allows in many cases to express higher-order moments in terms of lower-order moments. If, for instance, $\tilde{p}(z) = \frac{1+z}{1-z},$ then the system \eqref{eq:hierarchy} is written
\[ \frac{d}{dt} \mu_m =  m \mu_{m-1} - \frac{k^2 m^2}{2} \,\mu_m - m \mu_{m+1},
\]
and hence,
\[
  \mu_{m+1} =  \mu_{m-1} - \frac{k^2 m}{2} \, \mu_m - \frac{1}{m} \frac{d}{dt} \mu_m. 
\]

\subsection{Growth estimates}
First, we write the equation \eqref{eq:randomloewner} in polar coordinates using the notation $r = |\phi_t|,$ $\theta = \arg \phi_t,$ $U = \re p,$ $V = \im p:$
\begin{equation}
\begin{cases}
\label{eq:polarLoewner}
\frac{dr}{dt} = \left[ \cos(\theta-k B_t) (1 + r^2) -2r \right] U(r,\theta) + \sin(\theta-kB_t)(1-r^2)  V(r,\theta), \\
 \frac{d\theta}{dt} = - \left(\frac{1}{r}-r\right) \sin(\theta-kB_t) U(r,\theta) + \left[\left(\frac{1}{r} + r\right) \cos (\theta - kB_t) -2\right] V(r,\theta) .
\end{cases}
\end{equation}

Let the Herglotz function again be  given by $\tilde{p}(z) =  \frac{1+z}{1-z} \re \tilde{p}(0) + i \im \tilde{p}(0).$ Then \eqref{eq:polarLoewner} becomes
\begin{equation}
\begin{cases}
\label{eq:polarautomorphisms}
\frac{dr}{dt} = (1-r^2) |\tilde{p}(0)| \cos (\theta - k B_t - \arg \tilde{p}(0)),\\
 \frac{d\theta}{dt} = -2 \im \tilde{p}(0) - \left(\frac{1}{r} + r\right) |\tilde{p}(0)| \sin(\theta - k B_t - \arg \tilde{p}(0)).
\end{cases}
\end{equation}
Let $r_0 = |\phi_0(z)|=|z|.$ Then the solution to the first equation in \eqref{eq:polarautomorphisms} is given by
\begin{equation}
\label{eq:solutionautomorphisms}
\begin{cases}
r(t) = \tanh \left(|\tilde{p}(0)|\int_0^t \cos(\theta - k B_s - \arg \tilde{p}(0)) \, ds + \arctanh r_0\right), \quad r_0 < 1,\\
r(t) \equiv 1, \quad r_0 = 1.
\end{cases}
\end{equation}
In particular, this implies that $\phi_t(z)$ maps $\mathbb{T}$ onto $\mathbb{T},$ and the function $\tilde{p}(z) = \frac{1+z}{1-z} \re \tilde{p}(0) + i \im \tilde{p}(0)$  generates  families of automorphisms of $\mathbb{D}$ in the stochastic case, too.

From \eqref{eq:solutionautomorphisms} it is easy to obtain the following estimate
\begin{equation}
\label{eq:estcayley}
 \tanh ( - |\tilde{p}(0)| \, t + \arctanh r_0) \leq r(t) \leq \tanh (|\tilde{p}(0)| \, t + \arctanh r_0).
\end{equation}

Note that the estimate from below quickly becomes trivial as $t$ increases ($\tanh$ takes negative values, but $r(t)$ always stays non-negative).

In a similar way we can obtain estimates of growth of $r(t)$ for other Herglotz functions.

For $\tilde{p}(z) = \frac{1}{1-z}, $
\begin{equation}
\label{eq:estlinear}
 r_0\, e^{-t} - (1 - e^{-t}) \leq r(t) \leq r_0 \,e^{-t} + (1 - e^{-t}).
\end{equation}

For $\tilde{p}(z) = 1,$
\begin{equation}
\label{eq:estone}
 \frac{ r_0 \,(1-t) -t}{1 +  t\,(1+r_0)}  \leq r(t) \leq  \frac{r_0 \, (1-t) + t}{1 + t \,(1-r_0)}.
\end{equation}

It can be checked that the estimates from above in  \eqref{eq:estcayley} (in the case $\tilde{p}(z) = \frac{1+z}{1-z}$), \eqref{eq:estlinear} and \eqref{eq:estone} are sharp, and the equality is attained when $k=0$.

\subsection{Boundary diffusion}

Suppose again that the family $\phi_t(z)$ is a family of disk automorphisms,  $\tilde{p}(z) = \re \tilde{p}(0)\, \frac{1+z}{1-z} + \im \tilde{p}(0)\, i.$ Since the unit circle $\mathbb{T}$ is always mapped by $\phi_t(z)$ bijectively  onto itself, we can study the stochastic family of diffeomorphisms of the unit circle described by the equation
\[
 \frac{d\theta}{dt} = -2 \im \tilde{p}(0) - 2 \, |\tilde{p}(0)| \sin(\theta - k B_t - \arg \tilde{p}(0)),
\]
which is the restriction of the second equation in \eqref{eq:polarautomorphisms} to $\mathbb{T}.$

The process $\Theta_t = \arg \Psi_t  - \arg \tilde{p}(0)= \theta - k\,B_t -\arg \tilde{p}(0)$ is a diffusion on $\mathbb{T} \cong \mathbb{R} / 2 \pi \mathbb{Z}$:
\begin{equation}
\label{eq:noisyflow}
d \Theta_t =  -2 \, \left(\im \tilde{p}(0) + \, |\tilde{p}(0)| \sin \Theta_t\right) dt - k dB_t,
\end{equation}
with the infinitesimal generator given by
\[
 A = -2 \, \left(\im \tilde{p}(0) + \, |\tilde{p}(0)| \sin \theta\right) \frac{\partial}{\partial \theta }  +\frac{k^2}{2}\, dB_t. 
\]
The functions annihilating the infinitesimal generator, that is, the solutions to the equation $Af = 0$ are given by the formula
\[
 f(\theta) = c_1 + c_2 \int_0^\theta \exp \frac{4 \left(s \im \tilde{p}(0)- |\tilde{p}(0)| \cos s\right)}{k^2} ds
\]

The function $\tilde{p}(z) = \frac{1+z}{1-z}$ corresponds to the noisy North-South flow
\begin{equation}
\label{eq:noisycayleyflow}
d \Theta_t =  -2  \sin \Theta_t dt - k dB_t,
\end{equation}
studied in \cite{Carverhill86, Carverhill85} (see also \cite{RogersWilliams, Siegert}).

\appendix
\section{Complex It\^{o} formula and integration by parts}
\label{subseq:ComplexIto}
We recall  two basic results of stochastic calculus in the complex plane.

Let $Z_t = (X^1_t,Y^1_t,\ldots, X^n_t, Y^n_t)$ be a $2n$-dimensional real It\^o process, which may also be regarded as an $n$-dimensional complex It\^o process, by identifying $Z^k_t = X^k_t + i Y^k_t$. Let 
\[
f(z_1,\ldots, z_n) = (u_1(z_1,\ldots, z_n), v_1(z_1,\ldots, z_n), \ldots,u_m(z_1,\ldots, z_n), v_m(z_1,\ldots, z_n))
\]
be a $C^2$ function from $\mathbb{R}^{2n}$ to $\mathbb{R}^{2m}.$ Then it follows from the real multidimensional It\^{o} formula that $W_t = f(Z_t)$ is a $2m$-dimensional It\^o process, and its $p$-th (complex) component is given by
\begin{multline}
\label{eq:multiito}
dW_p = d(\re W_p + i \cdot \im W_p) = \sum_{j=1}^n \frac{\partial u_p}{\partial x_j} \, dX_j + \sum_{j=1}^n \frac{\partial u_p}{\partial y_j} \, dY_j \\ + \frac12 \sum_{j,k}\frac{\partial^2 u_p}{\partial x_j \partial x_k} \, dX_j \cdot dX_k + \sum_{j,k} \frac{\partial^2 u_p}{\partial x_j \partial y_k} dX_j \cdot dY_k + \frac12 \sum_{j,k}\frac{\partial^2 u_p}{\partial y_j \partial y_k} \, dY_j \cdot dY_k 
\\ + i \left(   \sum_{j=1}^n \frac{\partial v_p}{\partial x_j} \, dX_j + \sum_{j=1}^n \frac{\partial v_p}{\partial y_j} \, dY_j \right. 
\\ + \left. \frac12 \sum_{j,k}\frac{\partial^2 v_p}{\partial x_j \partial x_k} \, dX_j \cdot dX_k + \sum_{j,k} \frac{\partial^2 v_p}{\partial x_j \partial y_k} \, dX_j \cdot dY_k + \frac12 \sum_{j,k}\frac{\partial^2 v_p}{\partial y_j \partial y_k} \, dY_j \cdot dY_k  \right),
\end{multline}
where the dot product of stochastic differentials is a linear operation computed according to the following formal rules $dB_i \cdot dB_j =\delta _{ij} dt,$ $dB_i \cdot dt = dt \cdot dB_i = 0.$ 

 If we use the notations $\frac{\partial f}{\partial z_k} = \frac12 \left(\frac{\partial f}{\partial x_k} - i \frac{\partial f}{\partial y_k}\right)$ and $\frac{\partial f}{\partial \bar{z}_k} = \frac12 \left(\frac{\partial f}{\partial x_k} + i \frac{\partial f}{\partial y_k}\right),$ then \eqref{eq:multiito} simplifies to the well-known complex It\^{o} formula (see, e.~g., \cite{Uboe87}):

\begin{multline}
\label{eq:complexito}
 dW_p = \sum_{j=1}^n \frac{\partial f_p}{\partial z_j} \,dZ_j +   \sum_{j=1}^n \frac{\partial f_p}{\partial \bar{z}_j} \, d\bar{Z}_j 
\\ + \frac12 \sum_{j,k}^n \frac{\partial^2 f_p}{\partial z_i \partial z_j}\, dZ_i  \cdot dZ_j+ \frac12 \sum_{j,k}^n \frac{\partial^2 f_p}{\partial \bar{z}_i \partial \bar{z}_j}\, d\bar{Z}_i \cdot d\bar{Z}_j + \sum_{j,k}^n \frac{\partial^2 f_p}{\partial z_i \partial \bar{z}_j}\, dZ_i \cdot d\bar{Z}_j.
\end{multline}

Let $Z_t,$ $W_t$ be two It\^o processes in the complex plane. Take $f(z,w) = z \cdot w.$ Since all terms containing $\frac{\partial }{\partial \bar{z}}$ and $\frac{\partial }{\partial \bar{w}}$  vanish, we obtain the integration by parts formula for complex It\^o processes:

\[
 d(Z_t W_t) = W_t d Z_t + Z_t d W_t + dZ_t \cdot d W_t.
\]

In general, if the process $Z_t$ takes values in a domain $D \subset \mathbb{C}^n$  where the function $f$ is holomorphic, the formula \eqref{eq:complexito} reduces to
\[
 dW_p = \sum_{j=1}^n \frac{\partial f_p}{\partial z_j} \,dZ_j  + \frac12 \sum_{j,k}^n \frac{\partial^2 f_p}{\partial z_j \partial z_k}\, dZ_j \cdot dZ_k.
\]

\bibliographystyle{plain}

\end{document}